\documentclass[12pt]{article}

\usepackage[american]{babel}
\usepackage{amsmath,amsthm}
\usepackage{latexsym}
\usepackage{amssymb}
\usepackage[shortlabels]{enumitem}
\usepackage{tikz}
\usepackage{tikz-cd}
\theoremstyle{plain}
\newtheorem{Thm}{Theorem}
\newtheorem{Cor}[Thm]{Corollary}

\newtheorem{Lem}[Thm]{Lemma}
\newtheorem{KLem}[Thm]{Key Lemma}
\newtheorem{Prob}[Thm]{Problem}
\theoremstyle{definition}

\newtheorem{Facts}[Thm]{Facts}

\newtheorem{Def}[Thm]{Definition}

\allowdisplaybreaks

\newcommand{\bi}{\boldsymbol{i}}

\renewcommand{\phi}{\varphi}
\renewcommand{\rho}{\varrho}
\newcommand{\Cen}{\mathrm{Cen}}

\newcommand{\fm}{\mathfrak  m}

\newcommand{\fh}{\mathfrak  h}

\newcommand{\RR}{\mathbb{R}}

\newcommand{\CC}{\mathbb{C}}

\newcommand{\id}{\mathrm{id}}
\newcommand{\SKIP}[1]{}

\newcommand{\fg}{\mathfrak{g}}

\renewcommand{\emptyset}{\varnothing}

\begin{document}

\title{{\bf A note on commutators in compact semisimple Lie algebras}}
%\\\ \\
\author{Linus Kramer\thanks{Funded by the Deutsche Forschungsgemeinschaft through a Polish-German
\emph{Beethoven} grant KR1668/11, and under
Germany's Excellence Strategy EXC 2044-390685587, Mathematics M\"unster: Dynamics-Geometry-Structure.}}
\date{Dedicated to Jacques Tits}
\maketitle

\begin{abstract}
 Given any two elements $A,B$ in a compact semisimple Lie algebra, we show that there exist
 elements $X,Y,Z$ such that 
 \[
  A=[X,Y]\text{ and }B=[X,Z].
 \]
The proof uses Cartan subalgebras and their root systems. We also review some related
problems about Cartan subalgebras in compact semisimple Lie algebras.
\end{abstract}

Got\^o's Commutator Theorem \cite{Goto} \cite[6.56]{Compact} states that in a compact connected semisimple Lie group $G$, every element is a commutator. 
There is an infinitesimal version of Got\^o's Theorem which says that every element in a compact semisimple
Lie algebra $\fg$ is
a commutator, cp.~\cite[Thm.~A3.2]{ProLie}. The proof given in \emph{loc.cit}, which
uses Kostant's Convexity Theorem, is attributed to K.-H. Neeb.
Other proofs were given later by D'Andrea--Maffei and Malkoun--Nahlus \cite{AM,MNArx,MNLie}.
We prove the following somewhat stronger result by elementary means.
\begin{Thm}\label{Th1}
Let $\fg$ be a semisimple compact Lie algebra and let $A,B\in\fg$. Then
there is a regular element $X\in\fg$ with 
\[
 A,B\in[X,\fg]=\mathrm{ad}(X)(\fg).
\]
\end{Thm}
Our Lemma~\ref{L2}, which is the main step of the proof, uses a variant of Jacobi's method,
cp.~\cite{KHH}, \cite[App.~B]{MNArx}\footnote{Note that \cite{MNArx} and \cite{MNLie}
differ considerably.} and \cite{Wild}.
In the course of the proof we show in Corollary~\ref{C1} that every linear subspace $W\subseteq\fg$ of codimension
at most $2$ contains a Cartan subalgebra.
\begin{Def}
A finite dimensional real semisimple Lie algebra $\fg$ is called \emph{compact} if 
its Killing form $\langle-,-\rangle$ is negative definite. In this case
its adjoint group
\[
 G=\langle\exp(\mathrm{ad}(X))\mid X\in\fg\rangle
\]
is compact and
\[|X|=\sqrt{-\langle X,X\rangle}\] is a $G$-invariant euclidean norm on $\fg$.
In what follows,
orthogonality in $\fg$ will always refer to the Killing form.
The \emph{centralizer} of $A\in\fg$ is the Lie subalgebra \[\Cen_\fg(A)=\{X\in\fg\mid [X,A]=0\}.\]
\end{Def}
\begin{Lem}\label{L1}
Let $\fg$ be a compact semisimple Lie algebra and let $A\in\fg$. Then
$\fg$ decomposes (as a $\Cen_\fg(A)$-module) orthogonally as
\[
 \fg=\Cen_\fg(A)\oplus[A,\fg].
\]
\end{Lem}
\begin{proof}
Let $X,Y\in\fg$. If $X$ centralizes $A$, then 
\[
 \langle X,[A,Y]\rangle=\langle [X,A],Y\rangle=0,
\]
whence $X\in[A,\fg]^\perp$. 
Conversely, if $X\in[A,\fg]^\perp$, then 
\[
 0=\langle X,[A,Y]\rangle=\langle[X,A],Y\rangle
\]
holds for all $Y$ and thus $[X,A]=0$.
The Jacobi identity shows that $[X,[A,\fg]]\subseteq[A,\fg]$ for $X\in\Cen_\fg(A)$.
\end{proof}
We recall some facts about the structure of compact semisimple Lie algebras,
which can be found in \cite{Adams,BtD,He,HiNe,Compact}.
\begin{Facts}
Let $\fg$ be a compact semisimple Lie algebra.
We call a maximal abelian subalgebra $\fh$ of $\fg$ a \emph{Cartan subalgebra} or
a CSA for short. All CSAs in $\fg$ are conjugate under the action of $G$, cp.~\cite[V.6.4]{He} or \cite[6.27]{Compact}.
The dimension of $\fh$ is called the \emph{rank} of $\fg$.
Let $\fh\subseteq\fg$ be a CSA. Then 
\[
 T=\{\exp(\mathrm{ad}(H))\mid H\in\fh\}
\]
is a \emph{maximal torus} in $G$. As a $T$-module, the 
Lie algebra $\fg$ decomposes as an orthogonal direct sum of irreducible $T$-modules
\[
 \fg=\fh\oplus\sum_{\alpha\in\Phi^+}L_\alpha,
\]
cp.~\cite[Ch.~6]{Compact}.
The \emph{positive real roots} $\alpha\in\Phi^+$ are certain nonzero linear forms $\alpha:\fh\longrightarrow\RR$.
Each $T$-module $L_\alpha$ is $2$-dimensional and carries a complex structure $\bi$
such that $L_\alpha\cong\CC$ and 
\[
 \exp(\mathrm{ad}(H))(X)=\exp(2\pi\bi\alpha(H))X
\]
holds for all $H\in\fh$, $\alpha\in\Phi^+$ and $X\in L_\alpha$.
Hence $H\in\fh$ acts on $L_\alpha$ as 
\[
 \mathrm{ad}(H)(X)=[H,X]=2\pi\bi\alpha(H)X.
\]
The positive real roots separate the points in $\fh$, i.e. $\bigcap\{\mathrm{ker}(\alpha)\mid\alpha\in\Phi^+\}=\{0\}$.
The centralizer of an element $H\in\fh$ is therefore
\[
 \Cen_\fg(H)=\fh\oplus\sum_{\alpha(H)=0}L_\alpha.
\]
Hence $\Cen_\fg(H)=\fh$ holds if and only if $\alpha(H)\neq0$ for all 
positive real roots $\alpha$. Such elements $H$ are called \emph{regular}.
\end{Facts}
\begin{Lem}\label{so3}
Let $\fg$ be a compact semisimple Lie algebra, with a CSA $\fh$ and the corresponding decomposition
\[
 \fg=\fh\oplus\sum_{\alpha\in\Phi^+}L_\alpha
\]
as above, and let $\gamma\in\Phi^+$ be a positive real root.
Let $H_\gamma\in\fh$ be a nonzero vector orthogonal to $\mathrm{ker}(\gamma)$.
Then \[\fm_\gamma=\RR H_\gamma\oplus L_\gamma\cong\mathfrak{so}(3)\]
is the Lie algebra generated by $L_\gamma$.
\end{Lem}
\begin{proof}
We let $\fm_\gamma$ denote the Lie algebra generated by $L_\gamma$.
The centralizer of $\mathrm{ker}(\gamma)$ is $\fh\oplus L_\gamma$, whence $\fm_\gamma\subseteq\fh\oplus L_\gamma$.
Let $X\in L_\gamma$ be an element of norm $|X|=1$. Then $X,\bi X$ is an orthonormal basis for $L_\gamma$, and we put 
$Y=[X,\bi X]$. Then \[\langle X,Y\rangle=\langle[X,X],\bi X\rangle=0=\langle X,[\bi X,\bi X]\rangle=\langle Y,\bi X\rangle\] and thus 
$Y\in\fh$. For $H\in\fh$ we have 
\[
 \langle H,Y\rangle=\langle [H,X],\bi X\rangle=2\pi\gamma(H)\langle\bi X,\bi X\rangle=-2\pi\gamma(H),
\]
hence $Y$ is nonzero and orthogonal to $\mathrm{ker}(\gamma)$.
Moreover, $\langle Y,Y\rangle=-2\pi\gamma(Y)<0$. 
If we put $\rho=\frac{1}{\sqrt{2\pi\gamma(Y)}}$ and $U=\rho X,V=\rho \bi X,W=\rho^2Y$,
then \[ [U,V]=W,\ [V,W]=U,\ [W,U]=V\] and thus $\fm_\gamma\cong\mathfrak{so}(3)$.
\end{proof}

\begin{KLem}\label{L2}
Let $\fg$ be a compact semisimple Lie algebra and let 
$A,B\in\fg$. Suppose that $A$ is orthogonal to some CSA. 
Then there exists a CSA $\fh\subseteq\fg$ which
is orthogonal both to $A$ and to $B$.
\end{KLem}
\begin{proof}
Among all CSAs $\fh$ orthogonal to $A$, we choose one for which the orthogonal
projection $B_0$ of $B$ to $\fh$ has minimal length $r=|B_0|$. We claim that $r=0$.
Assume towards a contradiction that this is false. We decompose $\fg$ orthogonally
as 
\[
 \fg=\fh\oplus\sum_{\alpha\in\Phi^+}L_\alpha.
\]
Accordingly we have $A=\sum_{\alpha}A_\alpha$ and $B=B_0+\sum_\alpha B_\alpha$,
with $A_\alpha,B_\alpha\in L_\alpha$.
By assumption, $B_0\neq0$. Hence there is a positive real root $\gamma\in\Phi^+$ with $\gamma(B_0)\neq0$.
We decompose $B_0$ further in $\fh$ as an orthogonal sum $B_0=B_{00}+H_\gamma$,
where $\gamma(B_{00})=0$ and $H_\gamma\neq0$.
Then 
\[\fh=\RR H_\gamma\oplus\mathrm{ker}(\gamma)\] and 
\[\fm_\gamma=\RR H_\gamma\oplus L_\gamma\cong\mathfrak{so}(3)\]
by Lemma~\ref{so3}.
In the $3$-dimensional Lie algebra $\fm_\gamma\cong\mathfrak{so}(3)$ there is a 1-dimensional subspace $V\subseteq\fm_\gamma$ which 
is orthogonal to $H_\gamma$ and to $A_\gamma$. 
The adjoint representation of $\mathrm{SO(3)}$ on its Lie algebra
$\mathfrak{so}(3)$ is transitive on the 1-dimensional subspaces.
Hence there is an element $g\in G$ 
of the form $g=\exp(\mathrm{ad}(Z))$, for some $Z\in\fm_\gamma$,
with $g(H_\gamma)\in V$. 
Moreover, $g$ fixes $\mathrm{ker}(\gamma)$ pointwise.
The CSA $\fh'=g(\fh)=V\oplus\mathrm{ker}(\gamma)$ is then orthogonal to $A$.
The projection of $B$ to $\fh'$ is $B_{00}$ and has therefore strictly smaller length than $B_0$.
This is a contradiction.
\end{proof}
\begin{Cor}\label{C1}
Let $\fg$ be a compact semisimple Lie algebra and let $A,B\in\fg$. Then $A^\perp\cap B^\perp$
contains a CSA $\fh$.
\end{Cor}
\begin{proof}
We apply Lemma~\ref{L2} to $0$ and $A$ to obtain a CSA which is orthogonal to
$A$. Another application of Lemma~\ref{L2} to $A,B$ then yields a CSA $\fh$ which is orthogonal
to both $A$ and $B$. 
\end{proof}
\begin{proof}[Proof of Theorem~\ref{Th1}]
Let $\fh$ be a CSA which is orthogonal to $A$ and to $B$ and
let $X\in\fh$ be a regular element. Then $\fh=\Cen_\fg(X)$
and thus $A,B\in[X,\fg]$ by Lemma~\ref{L1}.
\end{proof}

\section*{Some remarks and open problems.}

We close with some remarks and an open problem. 
Suppose that $\fh$ is a CSA in the compact Lie algebra $\fg$.
If we pick nonzero elements $Z_\alpha\in L_\alpha$, for every positive root
$\alpha$, and if we put $Z=\sum_{\alpha\in\Phi^+}Z_\alpha$,
then $\fh\cap \Cen_\fg(Z)=0$. Since $\Cen_\fg(Z)$ contains a CSA $\fh'$, this
shows that there exists a CSA $\fh'$ which intersects $\fh$ trivially.
However, one can do better. The following is shown in \cite{MNLie}.
\begin{Thm}[Malkoun-Nahlus]\label{MN}
Let $\fh$ be a CSA in a compact semisimple Lie algebra $\fg$. Then there exists a CSA
$\fh'\subseteq\fh^\perp$.
\end{Thm}
We reproduce the beautiful proof from \cite{MNLie}.
\begin{proof}
We may assume that $\fg\neq0$.
Let $w$ be a Coxeter element in the Weyl group $W=N/T$, where $T$ is the maximal
torus corresponding to $\fh$, and $N\subseteq G$ is the normalizer of $T$.
Then $W$ acts as a finite reflection group on $\fh$, and 
$1$ is not an eigenvalue of $w$ in this action, cp.~\cite[3.16]{Hum}.
We choose $X\in\fg$ with $w=\exp(\mathrm{ad}(X))T$ and we claim that every
CSA $\fh'$ containing $X$ is orthogonal to $\fh$. The linear endomorphism
$\exp(\mathrm{ad}(X))-\id_\fg$ of $\fg$ maps $\fh$ onto $\fh$, and 
\[\exp(\mathrm{ad}(X))-\id_\fg=\sum_{k=1}^\infty\frac{1}{k!}\mathrm{ad}(X)^k
=\mathrm{ad}(X)\sum_{k=1}^\infty\frac{1}{k!}\mathrm{ad}(X)^{k-1}.\] 
In particular, $\mathrm{ad}(X)(\fg)\supseteq\fh$.
Thus $\Cen_\fg(X)\subseteq\fh^\perp$ by Lemma~\ref{L1}.
\end{proof}
Christoph B\"ohm has explained to me the following remarkable result.
\begin{Thm}
The orthogonal Lie algebras $\mathfrak{so}(m)$, for $m\geq 3$, can be decomposed
as orthogonal direct sum of CSAs.
\end{Thm}
\begin{proof}
The rank of $\mathfrak{so}(m)$ is $r=\lfloor \frac{m}{2}\rfloor$, and the dimension of $\mathfrak{so}(m)$
is $n=\frac{m(m-1)}{2}$. 
We let $e_1,\ldots,e_m$ denote the standard basis of $\RR^m$, and we put $X_{i,j}=e_ie_j^T-e_je_i^T$.
Then the $X_{i,j}$ with $i<j$ form an orthonormal basis of $\mathfrak{so}(m)$.
Moreover, two distinct basis elements $X_{i,j},X_{k,\ell}$ commute if and only if $\{i,j\}\cap\{k,\ell\}=\emptyset$.
The standard CSA for $\mathfrak{so}(m)$ is spanned by $X_{1,2},X_{3,4},\ldots,X_{2r-1,2r}$.
The claim follows if we can partition the set $\mathcal T_m$ of all two-element subsets of $\{1,\ldots,m\}$
into $\frac{n}{r}$ subsets consisting of $r$ pairwise disjoint $2$-element subsets. The latter is possible
by the scheduling algorithm for round robin tournaments. 

An explicit construction of such a partition of $\mathcal T_m$ can be described as follows, cp. \cite[Ex.~36.2]{vLW}.
For odd $m\geq 3$ put \[M_k=\{\{i,j\}\mid i<j\text{ and }i+j\equiv 2k\pmod m\},\]
for $k=1,\ldots,m$. The $M_k$ partition $\mathcal T_m$ into $m$ subsets of cardinality $\frac{m-1}{2}$,
each consisting of pairwise disjoint $2$-element subsets.
From this we obtain also such a partition of $\mathcal T_{m+1}$ by putting $M_k'=M_k\cup\{\{k,m+1\}\}$.
\end{proof}
We cannot expect such a result for general compact semisimple Lie algebras.
For example, the compact semisimple Lie algebra $\fg=\mathfrak{so}(5)\oplus\mathfrak{so}(3)$ has dimension
$13$, hence such a decomposition cannot exist. The following question is thus very natural.
\begin{Prob}
Which compact semisimple Lie algebras $\fg$ can be decomposed as an orthogonal sum of CSAs?
\end{Prob}
The monograph \cite{KT} is devoted to the complex version of this problem.

For the Lie algebras $\mathfrak{su}(m)$, the problem can be rephrased as follows, using the Veronese embedding
of $\CC\mathrm{P}^{m-1}$.
To each unit vector $u\in\CC^m$, we may assign the selfadjoint projector \[P(u)=uu^*,\]
where $*$ denotes the conjugate-transpose, and its traceless part \[\textstyle P_0(u)=uu^*-\frac{1}{m}\id_{\CC^m}.\]
We note that $P(uz)=P(u)$ holds for all complex numbers $z$ with $|z|=1$.
Suppose that $u_1,\ldots,u_m$ is an orthonormal basis of $\CC^m$. Then 
the projectors $P(u_1),\ldots,P(u_m)$ commute, and the matrices 
$\bi P_0(u_1),\ldots,\bi P_0(u_m)$ span a CSA $\fh$ in $\mathfrak{su}(m)$.
Conversely, the CSA $\fh$ determines the set of subspaces $u_1\CC,\ldots,u_m\CC$
uniquely, since these are the fixed points of the maximal torus $T\subseteq\mathrm{PSU}(m)$ 
with Lie algebra $\fh$ in its action 
on complex projective space $\CC\mathrm{P}^{m-1}$. 
Hence $\fh$ determines the orthonormal basis $u_1,\ldots,u_m$ up to a permutation of vectors,
and up to multiplication of the basis vectors by complex numbers of norm $1$.

The Killing form for $\mathfrak{su}(m)$ is given by 
$\langle X,Y\rangle=2m\mathrm{tr}(XY)$. The CSAs $\fh$ and $\fh'$ provided by two 
orthonormal bases $u_1,\ldots,u_m,v_1,\ldots,v_m$ are thus orthogonal if and only if
\[\textstyle
 |\langle u_k,v_\ell\rangle|^2=\frac{1}{m}
\]
holds for all $k,\ell$. In this case, the two bases are called \emph{mutually unbiased}.
Such bases were considered in quantum mechanics by J. Schwinger \cite{Schw}.
The construction of mutually unbiased bases has interesting connections to finite geometry,
cp.~\cite{Kan1}, \cite{Kan2}, \cite{KThas1}, \cite{KThas2}.
It is an open problem in which dimensions $m$ there exist $m+1$ pairwise
mutually unbiased orthonormal bases. 
They are known to exist if $m$ is a prime power \cite{WF}, \cite{KR}. 
As we have seen, this question is equivalent
to the existence of an orthogonal decomposition of $\mathfrak{su}(m)$ into CSAs.
There is a related problem about MASAs in operator theory, cp.~\cite{Hag}.
It is presently an open problem if $\mathfrak{su}(6)$
admits an orthogonal decomposition into $7$ CSAs.

\subsection*{Acknowledgments}
I thank Christoph B\"ohm, Theo Grundh\"ofer, Karl Heinrich Hofmann and Karl-Hermann Neeb for helpful remarks.

\end{document}